\documentclass[12pt]{amsart}
\usepackage{amssymb,amsmath,amsthm,amssymb,mathrsfs}
\usepackage{mathtools}
\usepackage{graphicx}
\usepackage{comment}
\usepackage{hyperref}
\usepackage{enumerate}
\usepackage{enumitem}
\usepackage{color}
\usepackage{epstopdf}
\renewcommand{\epsilon}{\varepsilon}

\DeclareMathOperator{\dist}{dist}

\def\R{\mathbb{R}}

\def\N{\mathbb{N}}

\def\e{\epsilon}

\def\k{\kappa}

\def\ba #1\ea {\begin{align} #1\end{align}}
\def\bann #1\eann {\begin{align*} #1\end{align*}}
\def\ben #1\een {\begin{enumerate} #1\end{enumerate}}
\def\bi #1\ei {\begin{itemize}\renewcommand\labelitemi{--} #1\end{itemize}}

\newtheorem{theorem}{Theorem}
\newtheorem*{theorem*}{Theorem}
\newtheorem{lemma}[theorem]{Lemma}
\newtheorem{proposition}[theorem]{Proposition}

\newtheorem{corollary}[theorem]{Corollary}
\newtheorem{definition}[theorem]{Definition}

\newtheorem*{conjecture*}{Conjecture}

\newtheorem*{claim*}{Claim}

\newtheoremstyle{TheoremNum}
        {\topsep}{\topsep}              
        {\itshape}                      
        {}                              
        {\bfseries}                     
        {.}                             
        { }                             
        {\thmname{#1}\thmnote{ \bfseries #3}}
    \theoremstyle{TheoremNum}

\date{\today}

\title[Slingshot]{Compact curve shortening flow solutions out of non compact curves}
\author{Theodora Bourni}
\author{Martin Reiris}
\address{Department of Mathematics, University of Tennessee Knoxville, Knoxville TN, 37996-1320}
\email{tbourni@utk.edu}
\address{Facultad de Ciencias, Universidad de la Rep\'ublica, Montevideo, Uruguay}
\email{mareithu@gmail.com}

\begin{document}

\begin{abstract}
We construct a slingshot, that is a compact, embedded solution to curve shortening flow that comes out of a non compact curve and exists for a finite time. 
\end{abstract}

\maketitle

\section{Introduction}
A smooth one-parameter family $\{\Gamma_t\}_{t\in I}$ of immersed planar curves
$\Gamma_t\subset \R^2$ evolves by curve shortening flow if 
\begin{equation}\label{csf}
\frac{\partial\gamma}{\partial t}(u, t)=\vec \kappa(u, t)\,,\,\,\forall (u,t)\in \Gamma\times I\,,
\end{equation}
for some smooth family $\gamma:\Gamma\times I\to \R^2$ of immersions $\gamma(\cdot, t):\Gamma\to \R^2$ of $\Gamma_t$, and where $\vec\kappa(u, t)$ is the curvature vector of $\Gamma_t$ at the point $\gamma(u, t)$.

When $\Gamma_0$ is a smooth embedded compact curve, then by a famous theorem of Grayson \cite{GR87}, the solution of the curve shortening flow starting from $\Gamma_0$ exists on a maximal time interval $[0, T)$ and as $t\to T$ the solution converges to a round point. In the case when $\Gamma_0$ is additionally convex, this theorem was previously proved by Gage and Hamilton \cite{GaHa86}. Contrary to the compact case, when $\Gamma_0$ is not compact solutions to curve shortening flow starting from $\Gamma_0$ are not that well understood in general. The particular case of graphical solutions has been extensively studied in the work of Ecker and Huisken \cite{EckerHuisken89, EckerHuisken91}, who, among other things, showed that the flow of entire graphs exists for all times. In  \cite{CZ98}, K-S Chou and X-P Zhu, showed that that if the initial curve divides the plane into two regions of infinite area, then a solution exists for all time. For the case that one of the regions of the plane defined by the curve has finite area, they showed that, if additionally the curve has finite total absolute curvature, then a solution exists for a finite time equal to that area divided by $\pi$. Moreover, they showed uniqueness of solutions when the initial curve has ends that are representable as graphs over two semi-infinite lines.

In the present paper we want to construct compact solutions emanating from a non compact initial curve. More precisely, given $\Gamma_0$ a smooth embedded curve in $\R^2$, we want to construct a smooth family of compact embeddings 
\[
\gamma: S^1\times(0, T)\to \R^2
\]
that satisfy the curve shortening flow equation \eqref{csf}, and such that the curves $\Gamma_t=\gamma(S^1, t)$ converge to $\Gamma_0$ as $t\to 0$, in the sense that for any $\e>0$, there exists $t_\e$ such that $\Gamma_t$ is in an $\e$-neighborhood of $\Gamma_0$ for all $t\in (0, t_\e)$. Note that such a solution is different from the one constructed in  \cite{CZ98}, as in  \cite{CZ98} the family of solutions satisfying curve shortening flow is non-compact, that is the parameter space $\Gamma$ in \eqref{csf} is homeomorphic to $\R$.

We will consider a curve $\Gamma_0$ that satisfies the following:
\begin{itemize}
\item[(i)]
$\Gamma_0$ is a smooth embedded 1-manifold diffeomorphic to $(0,1)$ and it separates $\R^2$ into two regions, one of which has  finite area, which we denote by $A_0\in (0, \infty)$. 
\item[(ii)] $a+1<b$ and $c>0$ are real numbers such that 
$\Gamma_0\subset (a,\infty)\times (-c,c)$ and $\Gamma_0\cap ([b,\infty)\times (-c,c))$ is the union of two smooth graphs, ${ u^\pm} \in [b,\infty)\rightarrow \R$ with { $u^{+}$} positive and decreasing to zero at infinity and {$u^{-}$} negative and increasing to zero at infinity, and with the derivatives of {$u^\pm$} converging to zero at infinity, as in Figure \ref{fig1}.
\end{itemize}
Moreover, we will denote by $B(\Gamma_0,\epsilon)$ the $\e$ neighborhood of $\Gamma_0$, that is
\[
B(\Gamma_0,\epsilon):=\{p\in \mathbb{R}^{2}: {\rm dist}(p,\Gamma_0)<\epsilon\}\,.
\]
Our main theorem is the following

\begin{theorem} \label{main}Let $\Gamma_0$ be a curve satisfying the above hypotheses (i)-(ii). There exists a smooth solution $\gamma:{\rm S}^{1}\times(0,\frac{A_0}{2\pi})\rightarrow \mathbb{R}^{2}$ to the curve shortening flow \eqref{csf} such that for any $\epsilon>0$ there exists $t_\e>0$ such that $\Gamma_t\subset B(\Gamma_0,\epsilon)$ for $0<t<t_\e$. 
\end{theorem}

\begin{figure}[h]
\centering
\includegraphics[scale=0.6]{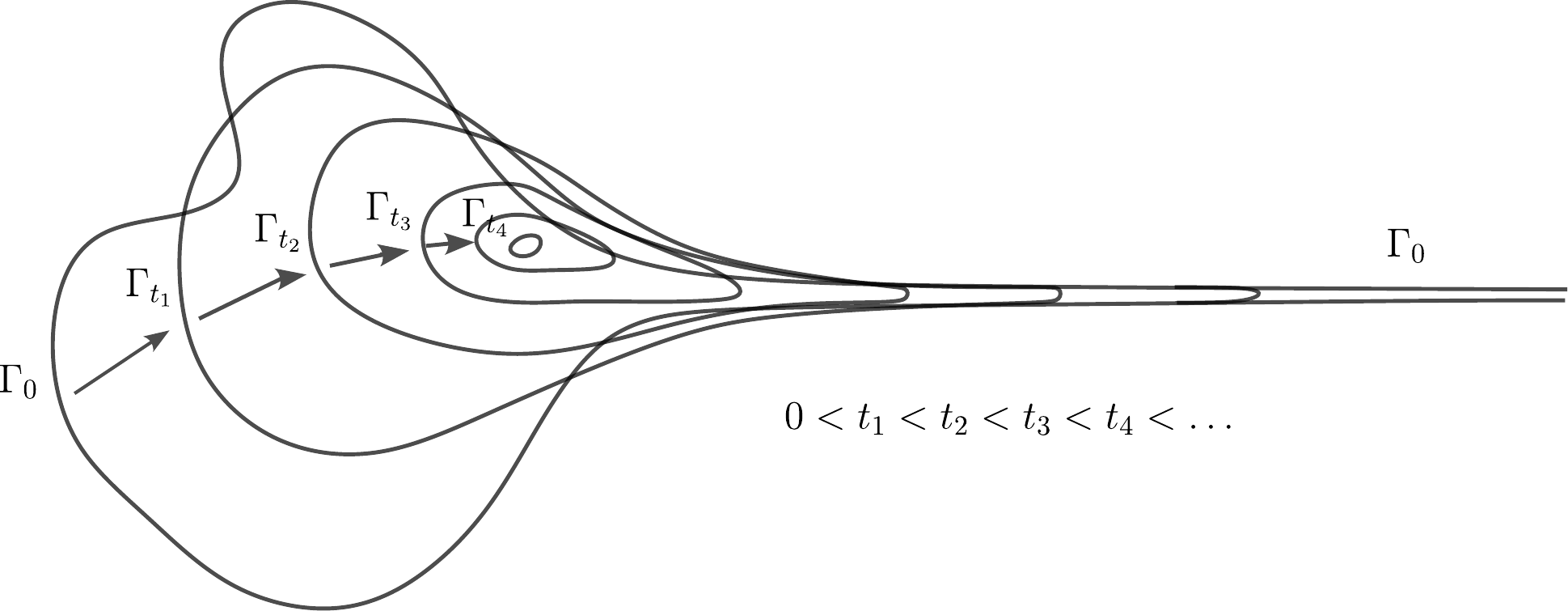}
\caption{Schematic figure of the evolution.}
\label{fig1}
\end{figure}

The construction of the solution described in Theorem \ref{main} is roughly as follows. We start with a sequence of compact curves $\Gamma_0^i$ that approximate $\Gamma_0$. Then, we define a sequence of curve shortening flows, using the curves $\Gamma_0^i$ as initial conditions, which we refer to as slingshots. The idea, then, is to show that one can extract a limit of these slingshots. To do this, we establish uniform curvature bounds for the slingshots away from the initial time 0. This argument, which is the most novel part of this construction, is a direct argument, based on repeated applications of the avoidance principle, and in particular the fact that the number of intersections between two solutions of curve shortening flow (at least one of which is compact) cannot increase in time\cite{MR953678}, together with the curvature estimates of Ecker and Huisken \cite{EckerHuisken91}.

\subsection*{Acknowledgements}

We would like to thank Facultad de Ciencias, Universidad de la Rep\'ublica in Montevideo, Uruguay, for hosting a visit of the first named author, during which this collaboration began. We also like to thank Sigurd Angenent and Mat Langford for conversations on the state of the art concerning non compact solutions to curve shortening flow.

TB was supported through grant 707699 of the Simons Foundation and grant DMS-2105026 of the National Science Foundation.

\section{Construction}
We first show that if a curve is locally, in some rectangle, a graph, then under curve shortening flow and in a smaller rectangle it remains a graph. Moreover, we obtain estimates on the gradient. We remark that such estimates are known in more general contexts but as the proof of the version we need here is relatively simple we do include it for the convenience of the reader.

\begin{proposition}\label{P1}
Let $\gamma_{0}:{\rm S}^{1}\rightarrow \mathbb{R}^{2}$ be a smooth embedding and suppose that for $D>0$, $R>0$ and $r<\frac D2$, the following holds:
\begin{enumerate}
\item for any $|x_{1}|\leq R$ and $|x_{2}|\leq R$, the segment joining $(x_{1},0)$ to $(x_{2},D)$ intersects {$\Gamma_{0}=\gamma_0(S^1)$} transversely and at just one point.
\item for any $|x|\leq R$, the balls $B_r((x,0))$ and $B_r((x,D))$ are disjoint from $\Gamma_{0}$.
\end{enumerate}

Then, the curve shortening flow solution $\gamma: S^1\times [0, T)$ starting at $\gamma(\cdot, 0)=\gamma_0(\cdot)$ satisfies $T\ge \frac{r^{2}}{2}$, and  for all $t\in [0, \frac{r^2}{2}]$ the timeslices $\Gamma_t$ satisfy the following: $\Gamma_t\cap ([-R,R]\times [0,D])$ can be represented as the graph of a smooth function $g_{t}: [-R,R]\rightarrow \mathbb{R}$, with 
\[
\sup_{x\in [-\frac R2, \frac R2]}|g'_t(x)|\le \frac{2D}{R}\,,\,\,and
\]
\[
\sqrt{r^2-2t}< g_t(x)< D-\sqrt{r^2-2t}\,,\,\,\forall x\in [-R, R]\,.
\]

\end{proposition} 

\begin{proof}
Note first that by hypothesis (2) of the proposition and the avoidance principle we obtain that 
\begin{equation}\label{endpoints}
([-R, R]\times\{0, D\})\cap \Gamma_t=\emptyset\,,\,\,\forall t\in[0, \tfrac{r^2}{2}]\,,
\end{equation}
and note that a simple linking argument shows that the curve shortening flow solution starting at $\Gamma_0$ does indeed have a lifespan of time at least $\frac{r^2}{2}$.
Recall that the number of intersections between two compact solutions of curve shortening flow cannot increase \cite{MR953678}. Therefore, hypothesis (1) of the proposition applied to segments with endpoints $(x, 0)$ and $(x, D)$, $x\in [-R,R]$, along with \eqref{endpoints}, imply that  $\Gamma_t\cap ([-\frac{R}{2},\frac{R}{2}]\times [0,D])$ can be represented as a graph of a smooth function $g_{t}: [-R,R]\rightarrow  \mathbb{R}$. To prove the gradient bound, consider a point on the graph $p=(x, g_t(x))$ with $x\in [-\tfrac R2, \tfrac R2]$ and suppose that $g_t(x)\ge \frac D2$. Consider the two line segments joining $(x\pm \frac{R}{2}, 0)$ to $p$ and extending them pass $p$ we note that they intersect the segment $[-R, R]\times\{D\}$. Thus, by hypothesis (1), these segments lie below the graph of $g_t$ and we obtain that $|g_t'(x)|\le \frac{g_t(x)}{R/2}\le \frac{2D}{R}$. If the point $p$ satisfies $g_t(x)\le \frac D2$, we obtain the same estimate by considering the segments joining $(x\pm \frac{R}{2}, D)$ to $p$ and extending them pass $p$. Finally, the height bounds are a cosequence of the avoidance principle and hypothesis (2).
\end{proof}

Proposition \ref{P1} and the curvature estimates of Ecker-Huisken  \cite{EckerHuisken91} yield the following 

\begin{corollary}\label{C1}
Under the hypothesis of Proposition \ref{P1}, for every integer $m\geq 1$, there is a constant $c_{m}=c(m,R,D, \Gamma_0)$ such that
\begin{equation}
\sup_{p\in \Gamma_t\cap ([-\frac{R}{4},\frac{R}{4}]\times[0, D])}|\partial_{s}^{m}\kappa(p,t)|\leq c_{m}\,,\,\,\forall t\in  [0,\tfrac{r^{2}}{2}]\,,
\end{equation}
where $\kappa(p,t)$ denotes the curvature of $\Gamma_t$ at the point $p$.
\end{corollary}
\begin{proof}
The proof is evident from the estimates in \cite{EckerHuisken91} by removing the time dependence from the bounds. Nonetheless, we include a sketch here for the convenience of the reader.

We first prove the case $m=0$. Consider a point $p_0=(x, y)$, with $|x|< \frac{R}{4}$ and $y\in (0, D)$, and let  $v=v(p,t)=\langle \nu, e_2\rangle^{-2}$, where $\nu=\nu(p,t)$ is a choice of the unit normal to $\Gamma_t$ at $p$. Consider now $G_t$ to be the connected component of $\Gamma_t\cap B_{\frac R4}(p_0)$ that is the graph of $g_t$  as in  Proposition \ref{P1}. Then, by Proposition \ref{P1}, we have that 
\[
v(p, t)\le 1+\frac{4D^2}{R^2}\,,\,\,\forall p\in G_t\,,\,\, \forall t\in [0, \tfrac{r^2}{2}]\,. 
\]
Define the function $g(p, t)=\k(p,t)^2\frac{v^2}{1- k^2v^2} ((\tfrac R4)^2-|p-p_0|^2)^2$, where $k=\frac12+\frac{2D^2}{R^2}$. Note that $g(p, 0)\le C R^2$, where $C=\sup_{G_0} \kappa^2$, a constant that depends only on $\gamma_0$.  If $g$ has a maximum at a point $(p, t)\in G_t\times (0, \frac{r^2}{2}]$, then, by computing the heat operator of $g$ (see \cite[proof of Theorem 3.1]{EckerHuisken91}), we obtain
\[
g(p,t)\le c(n,k) R^2\,.
\]
We therefore conclude the estimate for $m=0$.
The higher derivative bounds can be computed similarly by considering $\psi=1$ in \cite[proof of Theorem 3.4]{EckerHuisken91}. 
\end{proof}

\begin{definition}\label{def-approx}
A basic rectangle $\mathcal{F}(R,D,r)$ for an embedded curve $\Gamma$ consists of a number $r>0$ and a rectangle isometric to $[-R,R]\times [0,D]$ by an isometry $T$, such that: 
\begin{enumerate}
\item for any $|x_{1}|\leq R$ and $|x_{2}|\leq R$, the segment joining $T((x_{1},0))$ to $T((x_{2},D))$ intersects $\Gamma$ transversely and at just one point. 
\item for any $|x|\leq R$ the balls $B_{r}(T(x,0))$ and $B_{r}(T(x,D))$ are disjoint from $\Gamma$. 
\end{enumerate}
$T$ as above, will be referred to the isometry associated to $\mathcal{F}(R,D,r)$.

If $\mathcal{F}(R,D,r)$ is a basic rectangle for $\Gamma$ and $T$ is its associated isometry, then $T([-\frac{R}{4},\frac{R}{4}]\times [0,D])$ together with $r$, form also a basic rectangle for $\Gamma$, which will be denoted by $\mathcal{F}_{*}(R,D,r)$.
\end{definition}
\vspace{.2cm}

It is clear that the estimates in the statement of Corollary \ref{C1} work exactly the same when we replace the basic rectangle $[-R,R]\times [0,D]$ by basic rectangles $\mathcal{F}(R,D,r)$ for the curve $\Gamma_{0}$. More precisely, Proposition \ref{P1} and Corollary \ref{C1} yield the following:
\begin{proposition}\label{curvest} Assume that $\mathcal F(R, D,r)$ is a basic rectangle for an embedded smooth curve $\Gamma_0$. Then the curve shortening flow solution starting from $\Gamma_0$ exists for time at least $\frac{r^2}{2}$ and the timeslices $\Gamma_t$ satisfy the following curvature estimate.
For every integer $m\geq 1$, there is a constant $c_{m}=c(m,R,D, \Gamma_0)$, such that 
\begin{equation*}
\sup_{p\in \Gamma_t\cap \mathcal { F_{*}}(R,D,r)}|\partial_{s}^{m}\kappa(p,t)|\leq c_{m}\,,\,\,
\forall t\in  [0,\tfrac{r^{2}}{2}]\,,
\end{equation*}
where $\kappa(p,t)$ denotes the curvature of $\Gamma_t$ at the point $p$.
\end{proposition}
\vspace{.2cm}

\begin{definition}
For every integer $i\geq b+3$, consider the connected part of $\Gamma_0$ between $(i,u^{+}(i))$ and $(i,u^{-}(i))$ and cup it up with an embedded piece joining these two end points and lying inside the rectangle $[i,i+1]\times [u^{-}(i),u^{+}(i)]$, so that we obtain a smooth embedded and compact curve which we denote by $\Gamma_0^i$. Let $\gamma^{i}_{0}:{\rm S}^{1}\rightarrow \mathbb{R}^{2}$ be a parametrization of $\Gamma^i_0$. The solutions to the curve shortening flow starting from $\Gamma^{i}_{0}$ are denoted by $\Gamma^{i}_{t}$  and are called slingshots. Moreover, for each $i$, we will use $\gamma^i(\cdot, t)$ to denote any parametrization of the flow, which, as such, satisfies \eqref{csf}.
\end{definition} 

The following lemma says essentially that the slings enter compact regions in arbitrarily small times uniformly in $i$.

\begin{lemma}\label{P2} For any decreasing sequence of times $t_{j}\downarrow 0$, there exists a sequence of numbers $x_{j}$, such that the slingshots, after passing to a subsequence  $\Gamma_t^j$, satisfy
\[
\Gamma^{k}_{t}\subset [a, x_j]\times [-c,c]\,,\,\,\forall k\ge j,\text{ and }\, t\ge t_{j}\,.
\]
\end{lemma}

\begin{proof}

Consider a sequence $t_{j}\downarrow 0$. Then, by the assumptions on the initial curve $\Gamma_0$ and by construction of the approximating sequence $\Gamma^i_0$, the slingshots, after passing to a subsequence  $\Gamma_t^j$, satisfy the following. For any $j$, we can pick $x_j$ such that the following hold.
\begin{itemize}
\item[(i)] Let $\mathcal F(R, 2c, \sqrt{2t_j}):=[-R+ x_j, R+x_j]\times [-c, c]$, with $R=\frac{16c}{\pi}$. Then, for all $k\ge j$,  $\Gamma^k_0\cap \mathcal F(R, 2c, \sqrt{2t_j})$ has two connected components,  and for each of them $\mathcal F(R, 2c, \sqrt{2t_j})$  is a basic rectangle in the sense that on both components (i) and (ii) of Definition \ref{def-approx} are satisfied.
\item[(ii)] 
For all $k\ge j$, the area of the compact region bounded by $\Gamma^k_0$ in the halfplane $\{x\ge x_j-R\}$ is at most $\frac{\pi t_j}{2}$.
\end{itemize}
To prove the lemma, we will show that for all $j$ and $t\ge t_{j}$ we have $\Gamma^{k}_{t}\subset [a,R+x_{j}]\times [-c,c]$, for all $k\ge j$,  for which it suffices to prove that  $\Gamma^{k}_{t_j}\subset [a,R+x_{j}]\times [-c,c]$, for all $k\ge j$. Assume on the contrary that for some $ j$ and $k\ge j$ we have $\Gamma^{k}_{t_j}\cap( (R+x_{j}, \infty)\times [-c,c])\ne \emptyset$. 
First note that, by considering a small ball inside $\Gamma_0$ and by (i), the avoidance principle implies that  
\begin{equation*}\label{two}
\Gamma^k_t\cap \mathcal F(R, 2c, \sqrt{2t_j}) \text{ has two connected components, $\forall t\in [0, t_j]$}\,.
\end{equation*} 
 Let now $A^k_+(t)$ be the area of the compact region bounded by $\Gamma_t^k$ in the halfplane $\{x\ge x_j\}$. Since $\Gamma^k_t\cap \mathcal F(R, 2c, \sqrt{2 t_j})$ has two connected components, for all $t\in [0, t_j]$, Proposition \ref{P1} implies that 
\[
-\frac{d}{dt}A_+^k(t)\ge \pi-\frac{8c}{R}
\] 
and integration yields
\[
A_+^k(t_j)\le A_+^k(0)-t_j\left(\pi-\frac{8c}{R}\right)\le -\frac{\pi t_j}{2} +\frac{8c}{R}t_j<0
\]
which contradicts the hypothesis that $A_+^k(t_j)$ is positive, which is implied since we assumed that $\Gamma^{k}_{t_j}\cap( (x_{j}+R, \infty)\times [-c,c])\ne \emptyset$.
\end{proof}

The following lemma, which is the central lemma for our constructions, says that there is a decreasing sequence $t_{j}\downarrow 0$ such that  the slingshots, after passing to a subsequence $\Gamma^{j}_{t}$, for all $j$ and $t_{j}\leq t\leq t_{0}$ (where $t_0$ is some fixed positive time),  are covered by a fixed and finite set of basic rectangles and are therefore globally subject to the estimates of Corollary \ref{C1}. 
\begin{lemma}\label{mainlem}
There exists a decreasing sequence of times $t_{j}\downarrow 0$, $j\geq 0$, such that the slingshots, after passing to a subsequence $\Gamma_t^j$ satisfy the following. For every $j\geq 0$ there is a finite set of rectangles, 
\begin{equation}
\mathcal{F}(R_{j,1},D_{j,1},r_{j, 1}),\ldots,\mathcal{F}(R_{j,n_{j}},D_{j,n_{j}},r_{j,n_j}), 
\end{equation}
with $r_{j, k}\ge \sqrt{2t_0}$, $k=1, \dots, n_j$, that are basic for $\Gamma^{j}_{t}$ for any  $t\in [0,  t_{0}]$,
and moreover, 
\begin{equation}
\Gamma^{j}_{t}\subset \bigcup_{k=1}^{k=n_{j}}\mathcal{F}_{*}(R_{j,k},D_{j,k},r_{j,k})\,,\forall t\in [t_{j}, t_{0}].
\end{equation}
\end{lemma}

\begin{proof} 

We first construct basic rectangles that will cover the slingshots in a compact set, where all the initial curves $\Gamma_0^i$ coincide. 

Let $r_{0}>0$ be such that $[b,b+2]\times [0,c]$ and $[b,b+2]\times [-c,0]$ together with $r_{0}$ form basic rectangles for $\Gamma_0$, and we denote these by $\mathcal{F}^{\pm}$, respectively. 
Then, let
\begin{equation}
\mathcal{F}^{1}=\mathcal{F}(R_{1},D_{1},r_{1}),\ldots,\mathcal{F}^{l}=\mathcal{F}(R_{l},D_{l},r_{l}),
\end{equation}
be a collection of basic rectangles for $\Gamma_0$ with associated isometries $T_{m}$ and such that:
\begin{enumerate}[label=(\roman*)]
\item \label{(i)} $\mathcal{F}^{m}\subset \{x<b+2\}$, for $m=1,\ldots,l$,
\item \label{(iii)} $\mathcal{F}^{1}\subset {\rm Int}(\mathcal{F}^{+}_{*})$ and $\mathcal{F}^{l}\subset {\rm Int}(\mathcal{F}^{-}_{*})$,
\item \label{(iv)} $T_{m}(\{\frac{R_{m}}{4}\}\times [0,D_{m}])\subset {\rm Int}(\mathcal{F}^{m-1}_{*})$, for $m=2,\ldots, l$.
\end{enumerate}
Note that the rectangles $\mathcal{F}^{\pm}$ and $\mathcal{F}^{m}$, for $m=1,\dots, l$, are also basic rectangles for $\Gamma_0^i$, for all $i\in \N$. This is because they are contained in the half plane $\{x\le b+3\}$, where $\Gamma^{i}_{0}$ and $\Gamma_0$ coincide. Define, 
\begin{equation}\label{tb}
\bar{t}:=\tfrac12{\min\{r_{0}^{2},r^{2}_{1},\ldots,r^{2}_{l}\}} 
\end{equation}
and also 
\[
\mathscr{F}(0):=\{\mathcal F^+, \mathcal F^-, \mathcal F^1, \dots, \mathcal F^l\}\,.
\]
We claim that for any $i$ and $0\leq t\leq \bar{t}$ we have,
\begin{equation}\label{cover1}
\Gamma^{i}_{t}\cap \{x\leq b+5/4\}\ \subset\bigcup_{\mathcal F\in \mathscr F(0)}\mathcal{F}_{*}\,.
\end{equation}
To see this, let $\mathcal F\in \mathscr F(0)$.  Then, by Proposition \ref{P1}, we have that,  for any $i$ and any $0\leq t\leq \bar{t}$,  $\Gamma^{i}_{t}\cap \mathcal{F}_{*}$ is a connected 1-manifold with two boundary points lying in two opposite sides of the corresponding rectangle: $T_m(\{\pm \frac{R_m}{4}\}\times[0, D_m])$ if $\mathcal F=\mathcal F^m$, $m=1, \dots, l$, and accordingly if $\mathcal F=\mathcal F^\pm$. By conditions (ii) and (iii) above the claim follows.

\begin{figure}[h]
\centering
\includegraphics[scale=0.6]{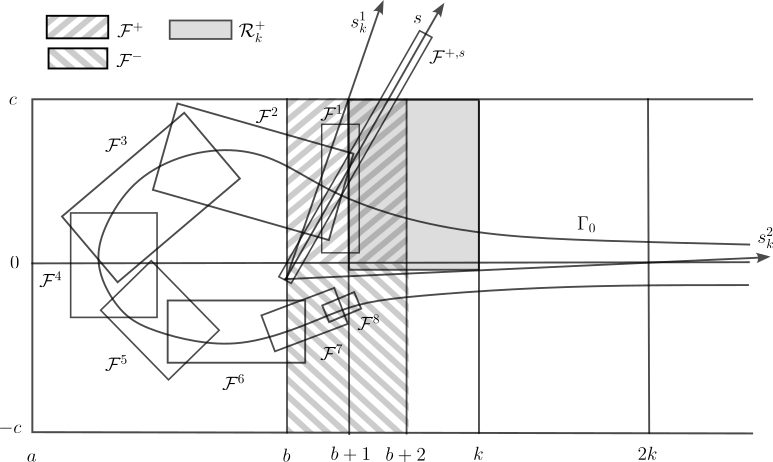}
\caption{Schematic figure of the rectangles $\mathcal{F}^{\pm}$, the $\mathcal{F}^{m}$ and the $\mathcal{F}^{+,s}$. The rectangle $\mathcal{R}^{+}_{k}$ is also shown.}
\label{fig2}
\end{figure}

The next step is to construct basic rectangles that cover the entirety of the slingshots for times $t> t_j$. An essential tool to do that is Lemma \ref{P2}, which allows us to deduce that after time $t_j$ all slingshots have entered a compact set.  

For any integer $k>b+1$,  we let $y_{k}:=\min\{u^{+}(2k),-u^{-}(2k)\}$ and set $q_k:=(b, -y_k)$. We then define $s^1_k$ and $s_k^2$ be the two rays starting from $q_{k}$ and passing through $(2k,0)$ and $(b+1, c)$ respectively. Note that both rays intersect $\Gamma_0$ transversely and only once at a point with positive $y$-coordinate. Define also the rectangle $\mathcal{R}^{+}_{k}:=[b+1,k]\times [\frac{-ky_{k}}{2k-b},c]$ and note that it lies in the region between the two rays and has one vertex on each of them.
Hence, any infinite ray from $q_{k}$ and passing through any point in $\mathcal{R}^{+}_{k}$ intersects $\Gamma_0$ transversely and only once. We will use this fact to cover the slingshots by basic rectangles in $\mathcal{R}^{+}_{k}$.


Let $\hat{r}\in (0, 1)$ be such that $\overline{B}_{\hat r}((b, 0))$ is contained in the open region of finite area enclosed by $\Gamma_0$. Since, $y_k\downarrow 0$ as $k\to \infty$, we can choose  $\hat{k}$ such that $y_{{k}}\leq \frac{\hat{r}}{4}$, for all $k\ge \hat k$, and from now on we consider such a $k\ge \hat k$.  Consider $s$ to be a ray starting from $q_{k}$ and passing through a point in $\mathcal{R}^{+}_{k}$. Since every such ray has positive slope and intersects $\Gamma_0$ transversely and only once, for each such $s$, we can find a rectangle $T_{s}([-R_{s},R_{s}]\times [0,D_{s}])$, for some isometry $T_{s}$, with the following properties:
\begin{enumerate}
\item\label{1} ${ T_{s}}(\{0\}\times [0,D_{s}])\subset s$ and  $T_{s}((0,0))=q_{k}$,
\item\label{2} $R_{s}\leq \frac{\hat{r}}{4}$ and $D_{s}$ is large enough so that $\langle T((0,D_{s})), e_2\rangle\geq c+\hat{r}$,
\item\label{3} $\Gamma_0\cap T_{s}([-R_{s},R_{s}]\times [0,D_{s}])$ is a graph over $T([-R_{s},R_{s}]\times \{0\})$. 
\end{enumerate}
Since $T_s([-R_{s},R_{s}]\times \{0\})\subset B_{\frac{\hat r}{2}}((b, 0))$ and by properties (2) and (3) above we conclude  that $T_{s}([-R_{s},R_{s}]\times [0,D_{s}])$ together with $r=\frac{\hat{r}}{4}$ is a basic rectangle for $\Gamma_0$, which we  denote  as $\mathcal{F}^{+,s}$. By compactness, we can find a collection of rays $s_{k,1},\ldots, s_{k,l_{k}}$ such that $\mathcal{F}^{+,s_{k,1}}_{*},\ldots,\mathcal{F}^{+,s_{k,l_{k}}}_{*}$ cover $\mathcal{R}_{k}^{+}$. From now on and to simplify notation we write $\mathcal{F}^{+,k,j}$ instead of $\mathcal{F}^{+,s_{k,j}}$.
An identical reasoning shows that we can find a collection of basic rectangles, 
\begin{equation}
\mathcal{F}^{-,k,1},\ldots,\mathcal{F}^{-,k,h_{k}},
\end{equation}
for $\Gamma_0$, all with $r=\frac{\hat{r}}{4}$ and such that $\mathcal{F}^{-,k,1}_{*},\ldots,\mathcal{F}^{-,k,h_{k}}_{*}$ are covering the rectangle $\mathcal{R}^{-}_{k}:=[b+1,k]\times [-c, \frac{ky_{k}}{2k-b}]$. We will denote by $\mathscr F(k)$ all these rectangles 
\begin{equation}\label{FK}
\mathscr F(k)=\{\mathcal{F}^{+,k,1},\ldots,\mathcal{F}^{+,k,l_{k}}, \mathcal{F}^{-,k,1},\ldots,\mathcal{F}^{-,k,h_{k}}\}\,.
\end{equation}
Note that $\mathcal{R}^{+}_{k}\cup \mathcal{R}^{-}_{k} = [b+1,k]\times [-c,c]$ and therefore
\begin{equation}\label{cover2}
 [b+1,k]\times [-c,c]\subset \bigcup_{\mathcal F\in \mathscr F(k)}\mathcal F_*\,.
\end{equation}

Given $k\geq \hat{k}$ let $\hat{i}_{k}>0$ be large enough so that none of the basic rectangles $\mathcal{F}\in \mathscr F(k)$ intersects the region $[\hat{i}_{k}, \infty]\times [-c,c]$. Note that this is possible, since all these rectangles have non zero slope and width bounded by $\frac{\hat r}{4}$. Recalling the definition of $\Gamma^{i}_{0}$, we deduce that these basic rectangles for $\Gamma_0$ are also basic rectangles for $\Gamma^{i}_{0}$ when $i\geq \hat{i}_{k}$.  Let $t_{j}\downarrow 0$ and $x_{j}$ be the sequences of Lemma \ref{P2}, for which, after dropping some initial terms if necessary, we will assume that $t_{1}< t_0:=\min\{\bar{t},\frac{\hat{r}^{2}}{32}\}$. 
Let $k_{1}$ be any integer such that $k_{1}\geq \max\{\hat{k}, x_{{1}}\}$. 
%
%
By Lemma \ref{P2}, we have that the slingshots, after passing to a subsequence $\Gamma_t^j$,  satisfy,  for any $j$ and $t\ge t_1$,
\begin{equation}\label{cover3}
\begin{split}
\Gamma^{j}_{t}\subset [a,x_{{1}}]\times [-c,c]&\subset  ([a,b+1]\times [-c,c])\cup \mathcal{R}^{+}_{k_{1}}\cup \mathcal{R}^{-}_{k_{1}}\\
&\subset \bigcup_{\mathcal F\in \mathscr F(0)\cup \mathscr F(k_1)}\mathcal F_*\,,
\end{split}
\end{equation}
with the second inclusion following by \eqref{cover1} and \eqref{cover2}, and where $\mathscr F(k)$ is as constructed in \eqref{FK}.
Finally note that for any $t\le t_0$ and for $i\ge \hat{i}_{k}$,  
$\mathcal F$ is a basic rectangle for $\Gamma^{i}_{t}$ for all $\mathcal F\in \mathscr F(0)\cup \mathscr F(k_1)$ and $t\in [0, t_0]$. Hence, the slingshots, after passing to a further subsequence, still denoted by $\Gamma_t^j$,  satisfy, for any $j$,
\[
\Gamma^j_t\subset \bigcup_{\mathcal F\in \mathscr F(0)\cup \mathscr F(k_1)}\mathcal F_*\,,\,\,\forall t\in [t_1, t_0],
\]
where $\mathscr F(0)\cup\mathscr F(k_1)$ is a finite family of rectangles that are basic for $\Gamma^j_t$, for all $j $ and $t\le t_0$ and moreover these rectangles are of the form $\mathcal F(R, D, r)$ with $r\ge \sqrt{2 t_0}$.
We can now finish the proof of the proposition, by constructing the rest of the sequence as follows. For each $t_j$ as above (from Lemma \ref{P2}), with $j\ge 2$, we choose $k_j\ge  \max\{k_{j-1}, x_j\}$ . Then we construct the family of basic rectangles $\mathscr F(k_j)$ as in \eqref{FK}. We then note that there exists $\hat i_{k_j}$ large enough,  so that none of the basic rectangles $\mathcal{F}\in \mathscr F(k_j)$ intersects the region $[\hat i_{k_j}, \infty]\times [-c,c]$, therefore for all $i\ge \hat i_{k_j}$ and $t\in [0, t_0]$, $ \mathcal F$  is a basic rectangle for $\Gamma^{i}_{t}$ for all $\mathcal F\in \mathscr F(0)\cup \mathscr F(k_j)$. Hence, the slingshots, after passing to a further subsequence, still denoted by $\Gamma_t^j$,  satisfy, for any $j$,
\[
\Gamma^j_t\subset \bigcup_{\mathcal F\in \mathscr F(0)\cup \mathscr F(k_j)}\mathcal F_*\,,\,\,\forall t\in [t_j, t_0],
\]
where $\mathscr F(0)\cup\mathscr F(k_j)$ is a finite family of rectangles that are basic for $\Gamma^j_t$, for $j\ge1 $ and $t\le t_0$ and moreover these rectangles are of the form $\mathcal F(R, D, r)$ with $r\ge \sqrt{2 t_0}$.

\end{proof}

\begin{proof}[Proof of Theorem \ref{main}]
Consider $t_0>0$ as in Lemma \ref{mainlem}. Lemma \ref{mainlem} and Proposition \ref{curvest} imply that we can apply a compactness argument  (which amounts to the Arzela--Ascoli theorem) 
 to the  sequence of embeddings $\gamma^j_{t_0}:S^1\to \R^2$. This yields that there exists a smooth embedding $\gamma^\infty_{t_0}:S^1\to \R^2$ and a sequence of diffeomorphisms of $S^1$, $\phi_j$, such that after passing to a subsequence, $\gamma^j_{t_0}\circ\phi_j$ converges smoothly to $\gamma^\infty_{t_0}$. Let $t_j\downarrow 0$ be as in Lemma \ref{mainlem} and define the diffeomorphisms

\[
\begin{split}
\psi_j: S^1\times [t_j, t_0]&\to S^1\times [t_j, t_0]\\
(x,t)&\mapsto \psi_j(x,t)=(\phi_j(x), t)\,.
\end{split}
\]
Note that Lemma \ref{mainlem} and Proposition \ref{curvest}, along with the evolution equation of the curvature and its derivatives (which yield time derivative bounds on the curvature and its derivatives), imply uniform bounds on the curvature and its derivatives  for the sequence $\gamma^j\circ\psi_j$ (locally in $S^1\times(0, t_0]$). Therefore, the Arzela--Ascoli theorem and a diagonal argument yield that there exists a smooth map $\gamma^\infty: S^1\times(0, t_0]\to\R^2$, with $\gamma^\infty(\cdot, t): S^1\to\R^2$ a smooth embedding for each $t\in (0, t_0]$ and $\gamma^\infty(\cdot, t_0)=\gamma^\infty_{t_0}(\cdot)$, and such that, after passing to a further subsequence, $\gamma^j\circ\psi_j$ converges to $\gamma^\infty$ smoothly on compact sets of $S^1\times(0, t_0]$. The smooth convergence does imply that $\gamma^\infty$ satisfies curve shortening flow \eqref{csf}. Also, since  $\gamma^\infty(\cdot, t): S^1\to\R^2$ a smooth embedding for each $t\in (0, t_0]$, by Grayson's theorem \cite{GR87}, we can extend the flow until it disappears to a round point. We have created thus a smooth flow  $\gamma^\infty: S^1\times(0, T)\to\R^2$, which agrees with the above defined $\gamma^\infty$ in $(0, t_0)$ and such that it converges to a round point as $t\to T$.

Finally, to finish the proof we need to show that
\begin{itemize}
\item[(i)] $T=\frac{A_0}{2\pi}$ and
\item[(ii)]  $\forall \epsilon>0$, $\exists \,t_\e>0$: $\Gamma_t:=\gamma^\infty(S^1, t)\subset B(\Gamma_0,\epsilon)$, $\forall \,0<t<t_\e$. 
\end{itemize}
To see (i), let $A^\infty(t)$ denote the (finite) area enclosed by $\Gamma_t$ and $A^j(t)$ that of the approximating curves $\Gamma^j_t=\gamma^j(S^1, t)$. By the convergence for $t\in (0, t_0]$, we have
\[
A^\infty(t_0)=\lim_j A^j(t_0)=\lim_jA^j(0)-2\pi t_0= A_0-2\pi t_0\,.
\]
Since $0=\lim_{t\to T}A^\infty(t)= A^\infty(t_0)-2\pi (T-t_0)$, we obtain (i).

In order to see (ii), we let $\e>0$. It suffices to show that there exists $t_\e$ such that for all $j$ large enough $\Gamma^j_t\subset B(\Gamma_0, \e)$, for all $t\in (0, t_\e)$. Assume that this is not the case, but instead, there exists a sequence of times $t_k\downarrow 0$ and a sequence of points of the slingshots $x_k\in \Gamma^{j_k}_{t_k}$, with $j_k\to \infty$,  such that $\dist(x_k, \Gamma_0)>\e$.
Note first, that by the assumption on $\Gamma_0$ and the approximating sequence $\Gamma_0^i$,  
a simple argument using grim reapers, parallel to the x-axis, as barriers implies that eventually the points $x_k$ must be in a compact set, that is, there exists $k_0$ and a compact set $K$, such that for all $k\ge k_0$, $x_k\in K$. 

Finally,  the proof of Lemma \ref{mainlem}, yields a uniform curvature bound for the slingshots in compact sets, which amounts to a uniform bound in the velocity. This implies that the distance traveled goes uniformly to zero, that is  $\dist(\Gamma^{j_k}_{t_k}\cap K, \Gamma_0)\to 0$, as $k\to \infty$, and thus we obtain a contradiction.

\end{proof}

\bibliographystyle{acm}
\bibliography{../../../bibliography}

\end{document}